\documentclass[a4paper,10pt]{amsart}
\usepackage[T1]{fontenc}
\usepackage{amsmath,amssymb,amsthm}
\usepackage{mathrsfs}

\newcommand{\ud}{\,\mathrm{d}}

\newcommand{\diam}[0]{\operatorname{diam}}

\newcommand{\abs}[1]{|#1|}
\newcommand{\Babs}[1]{\Big|#1\Big|}
\newcommand{\norm}[2]{|#1|_{#2}}
\newcommand{\bnorm}[2]{\Big|#1\Big|_{#2}}
\newcommand{\Norm}[2]{\|#1\|_{#2}}
\newcommand{\BNorm}[2]{\Big\|#1\Big\|_{#2}}
\newcommand{\pair}[2]{\langle #1,#2 \rangle}
\newcommand{\Bpair}[2]{\Big\langle #1,#2 \Big\rangle}

\newcommand{\bddlin}[0]{\mathscr{L}}

\newcommand{\reals}[0]{\mathbb{R}}
\newcommand{\complex}[0]{\mathbb{C}}

\newcommand{\ints}[0]{\mathbb{Z}}
\renewcommand{\emptyset}[0]{\text{\rm\O}}

\newcommand{\BMO}[0]{\operatorname{BMO}}
\newcommand{\cone}[0]{\Gamma}

\newcommand{\prob}[0]{\mathbb{P}}
\newcommand{\Exp}[0]{\mathbb{E}} 
\newcommand{\locally}[0]{\operatorname{loc}}
\newcommand{\radem}[0]{\varepsilon}

\newcommand{\rbound}[0]{\mathscr{R}}
\newcommand{\gbound}[0]{\mathscr{R}_{\gamma}}

\swapnumbers
\numberwithin{equation}{section}

\makeatletter
  \let\c@equation\c@subsection
\makeatother

\theoremstyle{plain}
\newtheorem{theorem}[subsection]{Theorem}
\newtheorem{proposition}[subsection]{Proposition}
\newtheorem{corollary}[subsection]{Corollary}
\newtheorem{lemma}[subsection]{Lemma}

\theoremstyle{definition}

\theoremstyle{remark}

\title[Banach space -valued BMO]{The Banach space -valued BMO, Carleson's condition, and paraproducts}
\author{Tuomas P.\ Hyt\"onen}
\address{Department of Mathematics and Statistics\\ University of Helsinki\\ Gustaf H\"allstr\"omin katu 2b\\ FI-00014 Helsinki\\ Finland}
\email{tuomas.hytonen@helsinki.fi}
\author{Lutz Weis}
\address{Institut f\"ur Analysis\\ Universit\"at Karlsruhe\\ Kaiserstra\ss{e} 89\\ D-76133 Karls\-ruhe\\ Germany}
\email{lutz.weis@math.uni-karlsruhe.de}
\date{\today}

\keywords{Bounded mean oscillation, square function, conical, Carleson embedding}
\subjclass[2000]{42B35 (primary); 42B20, 42B25, 46E40 (secondary)}

\begin{document}

\maketitle

\begin{abstract}
We define a scale of \(L^q\) Carleson norms, all of which characterize the membership of a function in BMO. The phenomenon is analogous to the John--Nirenberg inequality, but on the level of Carleson measures. The classical Carleson condition corresponds to the \(L^2\) case in our theory.

The result is applied to give a new proof for the \(L^p\)-boundedness of paraproducts with a BMO symbol. A novel feature of the argument is that all \(p\in(1,\infty)\) are covered at once in a completely interpolation-free manner. This is achieved by using the \(L^1\) Carleson norm, and indicates the usefulness of this notion. Our approach is chosen so that all these results extend in a natural way to the case of \(X\)-valued functions, where \(X\) is a Banach space with the UMD property.
\end{abstract}

\section{Introduction}

In order to describe properties of a function \(f\) on \(\reals^n\), it is often fruitful to consider its resolution \(F(x,t)=f*\psi_t(x)\), where \(\psi\) is an appropriate auxiliary function and \(\psi_t(x):=t^{-n}\psi(t^{-1}x)\). Then the value of \(F(x,t)\) tells something about the behaviour of \(f\) near the point \(x\in\reals^n\) and at the length scale of \(t\in(0,\infty)\).

It is well known that, for appropriate \(\psi\), the membership of \(f\) in various function spaces is encoded in a useful form in the size properties of \(F\), which typically involve some kind of a quadratic norm in the scale variable \(t\). For instance, the \(L^p\) norms in the reflexive range \(p\in(1,\infty)\) satisfy
\begin{equation}\label{eq:LpSq}\begin{split}
  \Norm{f}{L^p(\reals^n)}
  &\eqsim\Big(\int_{\reals^n}\Big[\int_0^{\infty}\abs{F(x,t)}^2\frac{\ud t}{t}\Big]^{p/2}\ud x\Big)^{1/p} \\
  &\eqsim\Big(\int_{\reals^n}\Big[\iint_{\abs{y-x}<t}\abs{F(y,t)}^2\frac{\ud y\ud t}{t^{n+1}}\Big]^{p/2}\ud x\Big)^{1/p},
\end{split}\end{equation}
where we refer to the first quadratic expression as a ``vertical'', and to the second as a ``conical'' square function, for obvious geometric reasons.

For the end-point space \(\BMO\) there holds in turn
\begin{equation}\label{eq:BMOSq}\begin{split}
  \Norm{f}{\BMO(\reals^n)}
  & \eqsim\sup_B\Big(\frac{1}{\abs{B}}\int_{B}\int_0^{r(B)}\abs{F(x,t)}^2\frac{\ud t}{t}\ud x\Big)^{1/2} \\
  & \eqsim\sup_B\Big(\frac{1}{\abs{B}}\int_{B}\iint_{\abs{y-x}<t<r(B)}\abs{F(y,t)}^2\frac{\ud y\ud t}{t^{n+1}}\ud x\Big)^{1/2},
\end{split}\end{equation}
where the supremum is taken over all balls \(B\subset\reals^n\) and \(r(B)\) is the radius of \(B\). Observe that, as opposed to \eqref{eq:LpSq} which really states two different non-trivial theorems, the equivalence of the ``vertical'' and ``conical'' forms in \eqref{eq:BMOSq} is completely elementary. The finiteness of the middle term in \eqref{eq:BMOSq} is the so-called Carleson condition for the measure \(\ud\mu(x,t)=\abs{F(x,t)}^2\ud x\ud t/t\).

When looking for generalizations of these results for Banach space -valued functions \(f:\reals^n\to X\), it has been known for some time that the quadratic norms should be reformulated in a ``randomized'' way (which we explain in more detail in Section~\ref{ss:square} below), and the Banach space \(X\) should satisfy the so-called unconditionality property of martingale differences (UMD). For such spaces, descriptions of the \(L^p(\reals^n;X)\) norm in terms of discrete versions of the randomized vertical square functions date back to the works of Bourgain~\cite{Bourgain} and McConnell~\cite{McConnell} in the 80's. More recently, also continuous-parameter quadratic estimates for \(f\in L^p(\reals^n;X)\) have been developed, both in the vertical \cite{Hytonen:LPS,KaiserWeis,KW} and the conical \cite{HNP} forms.

The aim of the present paper is to take up the BMO aspect of things in the vector-valued, continuous-parameter setting. (A discrete square function characterization of the vector-valued BMO has been previously given in \cite{Hytonen:H1} in terms of wavelet expansions.) While the two forms of the square function in \eqref{eq:BMOSq} are comparable, it will make a difference for vector-valued functions whether we take one or the other as the basis of the randomization procedure. Here we have chosen to work with the conical version, since it seems that this leads to the most satisfactory analogy with the classical results as presented e.g.~in \cite{Stein}.

As the exponent \(2\) plays no particular r\^ole for Banach space -valued functions \(F:\reals^{n+1}_+\to X\), we define in Section~\ref{sec:AC} a family of ``Carleson-type'' quantities \(C_q(F)(x)\), with \(x\in\reals^n\) and \(q\in(0,\infty)\), in such a way that \(C_2(F)(x)\), for a scalar-valued \(F\), reduces to the right side of \eqref{eq:BMOSq} but with the supremum only being over those balls \(B\) which contain \(x\). We can then prove the following characterization, extending the classical result (corresponding to \(X=\complex\) and \(q=2\)), which is found e.g.\ in \cite{Stein}, Theorem~IV.3.
The cases \(q\neq 2\) are apparently new even for \(X=\complex\); besides their appearance in this characterization, we believe that our new functionals \(C_q\) may find other uses in the theory of ``tent spaces'' of Coifman, Meyer and Stein \cite{CMS}; cf.\ the remark between Theorem~\ref{thm:AC} and its proof.

\begin{theorem}\label{thm:charBMO}
Let \(X\) be a UMD space. Let \(f\in L^1_{\locally}(\reals^n;X)\) with
\begin{equation*}
  \int_{\reals^n}\frac{\norm{f(x)}{X}}{(1+\abs{x})^{n+1}}\ud x<\infty,
\end{equation*}
and denote \(F(x,t):=f*\psi_t(x)\), where \(\psi\in\mathscr{S}(\reals^n)\) has vanishing integral.
\begin{itemize}
  \item If \(f\in\BMO(\reals^n;X)\), then \(C_q(F)\in L^{\infty}(\reals^n)\) for all \(q\in(0,\infty)\).
  \item If \(C_q(F)\in L^{\infty}(\reals^n)\) for some \(q\in(0,\infty)\) and \(\psi\) is non-degenerate in the sense of Section~\ref{ss:testf} below, then \(f\in\BMO(\reals^n;X)\).
\end{itemize}
Moreover, we have for all \(q\in(0,\infty)\) the equivalence of norms
\begin{equation*}
  \Norm{f}{\BMO(\reals^n;X)}
  \eqsim\Norm{C_q(F)}{L^{\infty}(\reals^n)}.
\end{equation*}
\end{theorem}

One of the key r\^oles played by the BMO space is in connection to the celebrated \(T(1)\) and \(T(b)\) theorems, and in particular in connection to the so-called paraproduct operators which make a decisive ingredient in the proof of these theorems. Let us define the paraproduct of the functions \(f\) and \(u\) first as the formal expression
\begin{equation*}
  P(f,u):=\int_0^{\infty}\psi_t*[(\psi_t*f)(\phi_t*u)]\frac{\ud t}{t},
\end{equation*}
where \(\psi,\phi\in\mathscr{S}(\reals^n)\) are fixed functions, the former one with a vanishing integral. As for the convergence issues, it is useful to think of \(P(f,u)\) as a linear functional on a suitable function space, its action on \(g\) being given by
\begin{equation}\label{eq:paraproduct}\begin{split}
  \pair{P(f,u)}{g}&:=\int_0^{\infty}\pair{\psi_t*[(\psi_t*f)(\phi_t*u)]}{g}\frac{\ud t}{t}.
\end{split}\end{equation}
The ``Carleson-type'' characterization of \(\BMO(\reals^n;X)\) from Theorem~\ref{thm:charBMO}, together with a conical square function description of \(L^p(\reals^n;X)\) from \cite{HNP}, allow a new and clean proof of the following basic mapping property of the paraproduct.

\begin{theorem}\label{thm:paraproduct}
Let \(X\) be a UMD space, and \(p\in(1,\infty)\). Let \(f\in\BMO(\reals^n;X)\) and \(u\in L^p(\reals^n)\). Then \(P(f,u)\in L^p(\reals^n;X)\) with
\begin{equation*}
  \Norm{P(f,u)}{L^p(\reals^n;X)}
  \lesssim\Norm{f}{\BMO(\reals^n;X)}\Norm{u}{L^p(\reals^n)},
\end{equation*}
in the sense that the integral in~\eqref{eq:paraproduct} converges absolutely for all \(g\in L^{p'}(\reals^n;X')\) and defines a linear functional on \(L^{p'}(\reals^n;X')\) with the mentioned norm bound.
\end{theorem}

Results of this flavour are already known for different versions of the vector-valued paraproduct; see \cite{Figiel,Hytonen:Tb,HMP,HW:T1}. (There is no canonical definition, even in the scalar-valued case, but the word `paraproduct' is generically used for various bilinear objects with a structure similar to \(P(f,u)\) above.) While the present proof has some elements in common with the previous ones, its advantage is the complete freedom from interpolation. We have also made an effort to choose our definitions and the set-up in such a way that the vector-valued theory parallels as much as possible the classical approach in the scalar case.

The rest of the paper is organized as follows: Section~\ref{sec:prelim} contains preliminary material and sets up some basic notation. In Section~\ref{sec:AC} we define and discuss our vector-valued versions of the \(A\) and \(C\) functionals of Coifman, Meyer and Stein \cite{CMS}, in terms of which we formulate our ``quadratic estimates''. Several basic results concerning these functionals are provided in Section~\ref{sec:ACest}. This preparation done, we are ready to prove Theorem~\ref{thm:charBMO} in Section~\ref{sec:charBMO}. In the final Section~\ref{sec:Carleson} we prove a version of the Carleson embedding theorem from which Theorem~\ref{thm:paraproduct} follows.

\subsection{Acknowledgement}
T.~Hyt\"onen is supported by the Academy of Finland (project 114374 ``Vector-valued singular integrals'').

\section{Preliminaries}\label{sec:prelim}

\subsection{Basic geometry}

The upper half-space is \(\reals^{n+1}_+:=\reals^n\times(0,\infty)\), whose points are usually denoted by \((x,t)\) or \((y,t)\) with \(x,y\in\reals^n\) and \(t\in(0,\infty)\). We write \(B(x,r):=\{y\in\reals^n:\abs{y-x}<r\}\) for the ball of centre \(x\) and radius \(r\), and given a ball \(B\), we denote by \(c(B)\) its centre and by \(r(B)\) its radius. The notation \(B\owns x\), e.g. in connection with a supremum, means that \(B\) runs over all (open) balls which contain the point \(x\).

A cube in \(\reals^n\) is a set of the form \(Q=x+[-h/2,h/2[^n\) where \(x\in\reals^n\) and \(h\in(0,\infty)\), and we write \(c(Q)=x\) and \(\ell(Q)=h\) for its centre and side-lenght, and \(\diam(Q)=\sqrt{n}\ell(Q)\) for the diameter. Given \(\alpha\in(0,\infty)\), a ball \(B\) and a cube \(Q\), we denote by \(\alpha B\) and \(\alpha Q\) the ball and the cube having the same centres as \(B\) and \(Q\), respectively, and \(r(\alpha B)=\alpha r(B)\), \(\ell(\alpha Q)=\alpha\ell(Q)\).

The cone of base \(x\in\reals^n\) and aperture \(\alpha\in(0,\infty)\), and its truncation at height \(h\in(0,\infty)\), are denoted by
\begin{equation*}
  \cone_{\alpha}(x) :=\{(y,t)\in\reals^{n+1}_+:\abs{y-x}<\alpha t\},\qquad
  \cone^h_{\alpha}(x) :=\{(y,t)\in\cone(x):t<h\}.
\end{equation*}
We also abbreviate \(\cone(x):=\cone_1(x)\) and similarly with the truncated version.

A \emph{Whitney decomposition} of an open set \(G\subset\reals^n\) with non-empty complement is a collection of disjoint cubes \(\{Q_j\}_{j=1}^{\infty}\) such that \(\bigcup_{j=1}^{\infty}Q_j=G\), and \(\diam(Q_j)<d(Q_j,G^c)\leq 4\diam(Q_j)\) for all \(j\). For instance we may choose \(\{Q_j\}_{j=1}^{\infty}\) to consist of all maximal cubes in \(\{Q_x:x\in G\}\), where \(Q_x\) is the smallest dyadic \(Q\owns x\) such that \(d(Q,G^c)\leq 4\diam(Q)\).

\subsection{``Square functions'' in Banach spaces}\label{ss:square}

Let us denote by \(\gamma_j\) a sequence of independent complex standard Gaussian random variables on some probability space \((\Omega,\mathscr{A},\prob)\).
Let \(H\) be a separable complex Hilbert space, and \(X\) a complex Banach space.

The \emph{Gauss space} \(\gamma(H,X)\) consists of those \(T\in\bddlin(H,X)\) such that for some (and then any) orthonormal basis \((e_j)_{j=1}^{\infty}\) of \(H\) the series \(\sum_{j=1}^{\infty}\gamma_j Te_j\) converges in \(L^2(\Omega;X)\). The norm in \(\gamma(H,X)\) is defined by
\begin{equation*}
  \Norm{T}{\gamma(H,X)}:=\BNorm{\sum_{j=1}^{\infty}\gamma_j Te_j}{L^2(\Omega;X)},
\end{equation*}
which is independent of the orthonormal basis \((e_j)_{j=1}^{\infty}\). See \cite{NW} for more details.

We can interpret $X$ as a subspace of $\gamma(H,X)$, and $\gamma(H,X)$ as a subspace of $L^2(\Omega;X)$ by identifying $x\in X$ with, say, the operator $h\mapsto x(h,e_1)$, and $T\in\gamma(H,X)$ with the function $\sum_{j=1}^{\infty}\gamma_j Te_j$.

Let then \(H=L^2(M,\mathscr{M},\mu)\). We say that \(f:M\to X\) is weakly \(L^2\) if \(\pair{f(\cdot)}{x'}:M\to\complex\) is in \(L^2(\mu)\) for all \(x'\in X'\). If this is the case, then \(f\cdot h:M\to X\) (pointwise product of functions) is Pettis-integrable for all \(h\in L^2(\mu)\), and there is a bounded operator (\cite{Pettis}, Theorem~3.4)
\begin{equation*}
  I_f\in\bddlin(H,X),\qquad I_f h
  :=\text{Pettis-}\int_{M}f\cdot h\ud\mu.
\end{equation*}
If it happens that \(I_f\in\gamma(H,X)\), then with slight abuse of notation we write \(f\in\gamma(H,X)\), and \(\Norm{f}{\gamma(H,X)}:=\Norm{I_f}{\gamma(H,X)}\). If \(I_f\notin\gamma(H,X)\), or if \(f\) is not even weakly \(L^2\), then we denote \(\Norm{f}{\gamma(H,X)}:=\infty\). The Gauss norm \(\Norm{f}{\gamma(H,X)}\) will be thought of as the ``square function'' of the Banach space -valued function \(f:M\to X\). Note that if \(X=\complex\), then  \(\Norm{f}{\gamma(H,X)}=\Norm{f}{H}=\Norm{f}{L^2(\mu)}\). See~\cite{NW} for further motivation of why this is a natural definition of a ``square function''. 

We need a result about pointwise multipliers on \(\gamma(H,X)\). 
A collection \(\mathscr{T}\subset\bddlin(X,Y)\) is called \emph{Gauss-bounded} if there holds
\begin{equation*}
  \Exp\BNorm{\sum_{j=1}^N\gamma_j T_j\xi_j}{Y}^2
  \leq C^2\Exp\BNorm{\sum_{j=1}^N\gamma_j\xi_j}{X}^2
\end{equation*}
for all \(N\in\ints_+\), \(\xi_1,\ldots,\xi_N\in X\) and \(T_1,\ldots,T_N\in\mathscr{T}\). The smallest admissible \(C\) is denoted by \(\gbound(\mathscr{T})\); if \(\mathscr{T}\) is the range of an operator-valued function \(g\), we also write \(\gbound(g):=\gbound(\mathscr{T})\). (In a number of other papers, the notation \(\gamma(\mathscr{T})\) has been used instead of \(\gbound(\mathscr{T})\); however, we decided to change this in order to avoid overloading the symbol \(\gamma\), and to emphasize the connection to the closely related notion of the \emph{Rademacher-bound}, which is defined with random signs \(\radem_j\) in place of the \(\gamma_j\) and denoted by \(\rbound(\mathscr{T})\).)

\begin{proposition}[\cite{KW}]\label{prop:multiplier}
Let \(f\in \gamma(H,X)\), and \(g:M\to\bddlin(X,Y)\) be strongly measurable with Gauss-bounded range. Then \(g\cdot f\in\gamma(H,Y)\) and \(\Norm{g\cdot f}{\gamma(H,Y)}\leq\gbound(g)\Norm{f}{\gamma(H,X)}\).
\end{proposition}

We often exploit this when \(g\) is a bounded scalar-valued function, which we identify with an \(\bddlin(X)\)-valued one via \(\lambda\leftrightarrow\lambda I\). Then it is well-known (and can be proved by a convexity argument) that \(\gbound(g)\leq\Norm{g}{\infty}\). In particular, when $v$ is a unimodular complex function, it follows that $\Norm{v\cdot g}{\gamma(H,X)}=\Norm{g}{\gamma(H,X)}$.

The following duality will also be needed:

\begin{proposition}[\cite{KW}]\label{prop:duality}
If $f\in\gamma(H,X)$ and $g\in\gamma(H,X')$ are strongly measurable functions on $M$, then their pointwise duality product is in $L^1(\mu)$, and satisfies
\begin{equation*}
  \int_M\abs{\pair{f(t)}{g(t)}}\ud\mu(t)\leq\Norm{f}{\gamma(H,X)}\Norm{g}{\gamma(H,X')}.
\end{equation*}
\end{proposition}

\subsection{Test functions}\label{ss:testf}

For convenience, we choose to work with test functions \(\psi,\phi\) in the Schwartz class \(\mathscr{S}(\reals^n)\), although an interested reader will easily check that much less would often suffice. We call a test function \(\psi\) \emph{degenerate} if its Fourier transform \(\hat{\psi}\) vanishes identically on some ray \(\{t\xi:t\in[0,\infty)\}\), where \(\xi\in\reals^n\setminus\{0\}\); otherwise it is called \emph{non-degenerate}. For a non-degenerate \(\psi\), one can always find a \emph{complementary function} \(\phi\in\mathscr{S}(\reals^n)\) such that
\begin{equation}\label{eq:complementary}
  \int_0^{\infty}\hat{\psi}(t\xi)\hat{\phi}(-t\xi)\frac{\ud t}{t}=1
\end{equation}
for all \(\xi\in\reals^n\setminus\{0\}\), and \(\hat{\phi}\) may be chosen to vanish in a neighbourhood of the origin, so that in particular \(\phi\) has vanishing integral. In fact, a possible \(\phi\) is given by
\begin{equation*}
  \hat{\phi}(-\xi):=\frac{\chi(\xi)\overline{\hat{\psi}(\xi)}}{\int_0^{\infty}\chi(t\xi)\abs{\hat{\psi}(t\xi)}^2\ud t/t}
\end{equation*}
where we have first chosen a non-negative \(\chi\in C_c^{\infty}(\reals^n\setminus\{0\})\) which vanishes in a sufficiently small neighbourhood of \(0\) so that \(\chi\cdot\hat{\psi}\) is still non-identically zero on all rays from the origin.

\section{The functionals \(A\) and \(C\) of Coifman, Meyer and Stein}\label{sec:AC}

Let \(F:\reals^{n+1}_+\to X\) be strongly measurable, i.e., a pointwise limit of simple \(X\)-valued functions. (For operator-valued functions \(G:\reals^{n+1}_+\to\bddlin(X,Y)\), which we also encounter later on, strong measurability means the strong measurability of \(G(\cdot)\xi:\reals^{n+1}\to Y\) for all \(\xi\in X\).) We now introduce vector-valued versions of the various functionals used by Coifman, Meyer and Stein~\cite{CMS} to define the norms of their ``tent spaces''. To simplify notation, we abbreviate
\begin{equation*}
  \gamma(X):=\gamma\Big(L^2\big(\reals^{n+1}_+,\frac{\ud y\ud t}{t^{n+1}}\big);X\Big).
\end{equation*}

According to the philosophy that the Gauss norms should replace the classical quadratic expressions in the Banach space -valued analysis,
our conical square function and its truncated version are defined by
\begin{equation*}
  A^{(\alpha)}(F)(x):=\Norm{F\cdot 1_{\cone_{\alpha}(x)}}{\gamma(X)},\qquad
  A^{(\alpha)}(F|h)(x):=\Norm{F\cdot 1_{\cone_{\alpha}^h(x)}}{\gamma(X)}.
\end{equation*}
Recall that we interpret the norms as being \(\infty\) if the quantities inside them do not belong to the indicated space.
We abbreviate \(A(F):=A^{(1)}(F)\), \(A(F|h):=A^{(1)}(F|h)\) for our most common choice of the aperture.

For \(X=\complex\), our \(A\) reduces to the functional \(A=A_2\) defined in \cite{CMS}, Section~1. While Coifman et al. consider a family of related functionals \(A_q\) with parameter \(q\in[1,\infty]\) (which may be naturally extended to \(q\in(0,\infty]\) as in \cite{CV}), we only define the vector-valued \(A\) in the quadratic case \(q=2\) as above, and in the end-point case \(q=\infty\). However, the natural setting for the latter one is somewhat different, and we also choose a different notation:

For an operator-valued function \(G:\reals^{n+1}_+\to\bddlin(X,Y)\) (with no additional conditions), we define the non-tangential maximal function
\begin{equation*}
  N^{(\alpha)}(G)(x):=\gbound(\{G(y,t):(y,t)\in\cone_{\alpha}(x)\}).
\end{equation*}
We also write \(N(G):=N^{(1)}(G)\).

\begin{lemma}
For any \(G\), the function \(N^{(\alpha)}(G):\reals^n\to[0,\infty]\) is lower semicontinuous.
\end{lemma}

\begin{proof}
Suppose that \(N^{(\alpha)}(G)(x)>\lambda\). Because of the finitary nature of Gauss-boundedness, this means that there exist \(N\in\ints_+\) and \((y_1,t_1),\ldots(y_N,t_N)\in\cone_{\alpha}(x)\) such that
\begin{equation*}
  \gbound(\{G(y_1,t_1),\ldots,G(y_N,t_N)\})>\lambda.
\end{equation*}
But these points also belong to \(\cone_{\alpha}(x')\) whenever \(\abs{x'-x}<\min_{i=1,\ldots,N}[\alpha t_i-\abs{y_i-x_i}]\), and hence \(N^{(\alpha)}(G)(x')>\lambda\) for all \(x'\) in this neighbourhood of \(x\). Thus \(\{N^{(\alpha)}(G)>\lambda\}\) is open.
\end{proof}

We then come to the \(C\) functional, which is relevant for the description of the \(\BMO\) space. It is defined as an average of the truncated \(A\) functionals, with an additional parameter \(q\in(0,\infty)\) indicating the type of the average:
\begin{equation*}
  C_q^{(\alpha)}(F)(x) :=\sup_{B\owns x}\Big(\frac{1}{\abs{B}}\int_B A^{(\alpha)}(F|r(B))^q(y)\ud y\Big)^{1/q}.
\end{equation*}
Note that by Jensen's (or H\"older's) inequality, \(C^{(\alpha)}_q (F)(x)\) is a nondecreasing function of \(q\in(0,\infty)\). Once again, we abbreviate \(C_q(F):=C_q^{(1)}(F)\).

If \(X=\complex\), then
\begin{equation*}\begin{split}
 C_2(F)^2(x)
 &=\sup_{B\owns x}\frac{1}{\abs{B}}\int_B\iint_{\cone^{r(B)}(y)}\abs{F(z,t)}^2\frac{\ud z\ud t}{t^{n+1}}\ud y \\
 &=\sup_{B\owns x}\frac{1}{\abs{B}}\int_0^{r(B)}\frac{\ud t}{t}
  \int_{2B}\ud z\abs{F(z,t)}^2\frac{\abs{B\cap B(z,t)}}{t^n} \\
 &\eqsim\sup_{B\owns x}\frac{1}{\abs{B}}\iint_{B\times(0,r(B))}\abs{F(y,t)}^2\frac{\ud y\ud t}{t}
\end{split}\end{equation*}
reduces to the definition of \(C=C_2\) in \cite{CMS} using the Carleson cylinders \(B\times(0,r(B))\). We note that in \cite{CMS} a scale of functionals \(C_q\) with \(q\in[1,\infty)\) (and naturally extended to \(q\in(0,\infty)\), cf.~\cite{CV}) is also defined, but it differs from our scale, with equivalence of the functionals only at \(q=2\).

\section{Basic estimates for the \(A\) and \(C\) functionals}\label{sec:ACest}

The \(A\) functionals have already been studied in the vector-valued context in \cite{HNP}, and we quote two inequalities from there. Both results involve the UMD property of the underlying Banach space \(X\), and our use of UMD in the present paper will be mainly via the application of these estimates.

\begin{theorem}[\cite{HNP}, Theorem~4.3]\label{thm:aperture}
Let \(X\) be a UMD space, and \(p\in(1,\infty)\) and \(\alpha\in(0,\infty)\).
For all strongly  measurable \(F:\reals^{n+1}_+\to X\), there holds
\begin{equation*}
  \Norm{A^{(\alpha)}(F)}{L^p(\reals^n)}
  \eqsim\Norm{A(F)}{L^p(\reals^n)}.
\end{equation*}
\end{theorem}

\begin{theorem}[\cite{HNP}, Theorem~4.8]\label{thm:LpTp}
Let \(X\) be a UMD space, and \(p\in(1,\infty)\). Let \(\psi\in\mathscr{S}(\reals^n)\) have a vanishing integral and \(F(x,t):=f*\psi_t(x)\) for \(f:\reals^n\to X\). Then
\begin{equation*}
  \Norm{A(F)}{L^p(\reals^n)}\lesssim \Norm{f}{L^p(\reals^n;X)}.
\end{equation*}
\end{theorem}

Next we give an end-point extension of the previous theorem:

\begin{corollary}\label{cor:H1T1}
Under the assumptions of Theorem~\ref{thm:LpTp}, there also holds
\begin{equation*}
  \Norm{A(F)}{L^1(\reals^n)}\lesssim \Norm{f}{H^1(\reals^n;X)},
\end{equation*}
where \(H^1\) is the real-variable Hardy space.
\end{corollary}

\begin{proof}
We observe that
\begin{equation}\label{eq:sandwich}\begin{split}
  A(F)(x):=\Norm{F\cdot 1_{\cone(x)}}{\gamma(X)}
  &\leq\Norm{(y,t)\mapsto F(y,t)\eta((y-x)/t)}{\gamma(X)} \\
  &\leq A^{(2)}(F)(x)
\end{split}\end{equation}
where \(\eta\in C_c^{\infty}(\reals^n)\) is bounded by one everywhere, equal to one in \(B(0,1)\), and vanishes outside \(B(0,2)\). It follows from Theorems~\ref{thm:aperture} and \ref{thm:LpTp} that the linear mapping taking \(f\) to \(x\mapsto[(y,t)\mapsto F(y,t)\eta((y-x)/t)]\) is bounded from \(L^p(\reals^n;X)\) to \(L^p(\reals^n,\gamma(X))\) for \(p\in(1,\infty)\). 
This mapping is given by the integral operator
\begin{equation*}
  f(x)\mapsto\int_{\reals^n}K(x,z)f(z)\ud z,
\end{equation*}
where
\begin{equation*}
  K(x,z):X\to\gamma(X),\xi\mapsto\big[(y,t)\mapsto \frac{1}{t^n}\phi(\frac{x-z}{t})\eta(\frac{y-x}{t})\big]\otimes\xi.
\end{equation*}
To obtain its boundedness from \(H^1(\reals^n;X)\) to \(L^1(\reals^n,\gamma(X))\) (which by \eqref{eq:sandwich} completes the proof), it suffices to show (see \cite{RdFRT}) that the \(K\) defined above is a Calder\'on--Zygmund kernel. Observe that
\begin{equation*}
  \Norm{\xi\mapsto h\otimes\xi}{\bddlin(X,\gamma(X))}
  =\Norm{h}{L^2(\tfrac{\ud y\ud t}{t^{n+1}})}.
\end{equation*}
Hence the claim follows from the computations
\begin{equation*}\begin{split}
  &\Norm{K(x,z)}{\bddlin(X,\gamma(X))}
  =\iint_{\reals^{n+1}_+}\Babs{\frac{1}{t^n}\phi(\frac{x-z}{t})
     \eta(\frac{y-x}{t})}^2\frac{\ud y\ud t}{t^{n+1}} \\
  &\leq\int_0^{\infty}\int_{B(x,ct)}\frac{1}{t^{2n}}\Big(1+\frac{\abs{x-z}}{t}\Big)^{-2(n+1)}\ud y\frac{\ud t}{t^{n+1}} 
   \lesssim\frac{1}{\abs{x-z}^{2n}},
\end{split}\end{equation*}
and
\begin{equation*}\begin{split}
  &\Norm{\nabla_z K(x,z)}{\bddlin(X,\gamma(X)^n)}
  \lesssim\iint_{\reals^{n+1}_+}\Babs{\frac{1}{t^{n+1}}\nabla\phi(\frac{x-z}{t})
     \eta(\frac{y-x}{t})}^2\frac{\ud y\ud t}{t^{n+1}} \\
  &\leq\int_0^{\infty}\int_{B(x,ct)}\frac{1}{t^{2(n+1)}}
    \Big(1+\frac{\abs{x-z}}{t}\Big)^{-2(n+2)}\ud y\frac{\ud t}{t^{n+1}}
   \lesssim\frac{1}{\abs{x-z}^{2(n+1)}}.
\end{split}\end{equation*}
In both cases the last step follows easily after observing that the \(y\)-integration only yields a factor \(Ct^n\), and splitting the \(t\)-integration to the two intervals \([0,\abs{x-z})\) and \([\abs{x-z},\infty)\).
\end{proof}

We then present some norm inequalities between \(A(F)\) and \(C_q(F)\). These estimates mostly depends on the definition of \(C_q(F)\) as an average of the truncated versions of \(A(F)\), and have quite little to do with the more precise internal structure of these quantities. Hence the proof of the following theorem is almost a repetition of the original scalar-valued argument of Coifman, Meyer and Stein~\cite{CMS}, and we only write it out for the sake of completeness.

\begin{theorem}[\cite{CMS}, Theorem~3]\label{thm:AC}
Let \(F:\reals^{n+1}_+\to X\) be strongly measurable.
The following relations are valid for all \(\alpha\in(0,\infty)\):
\begin{itemize}
  \item[(a)] If \(X\) is a UMD space, \(p\in(1,\infty)\) and \(q\in(0,\infty)\), then
\begin{equation*}
  \Norm{A^{(\alpha)}(F)}{L^p(\reals^n)}\lesssim \Norm{C_q^{(\alpha)}(F)}{L^p(\reals^n)}.
\end{equation*}
   \item[(b)] If \(X\) is any Banach space and \(0<q<p\leq\infty\), then
\begin{equation*}
  \Norm{C_q^{(\alpha)}(F)}{L^p(\reals^n)}\lesssim \Norm{A^{(\alpha)}(F)}{L^p(\reals^n)}.
\end{equation*}
\end{itemize}
In particular, if \(X\) is a UMD space and \(0<q\leq 1<p<\infty\), then for all \(\alpha,\beta\in(0,\infty)\) there holds
\begin{equation*}
  \Norm{A^{(\alpha)}(F)}{L^p(\reals^n)}\eqsim\Norm{C^{(\beta)}_q(F)}{L^p(\reals^n)}.
\end{equation*}
\end{theorem}

Note that the last statement of the theorem gives an indication of the usefulness of the functionals \(C_q\) for \(q\in(0,1]\), even in the scalar setting \(X=\complex\). Had we insisted on the use of \(C_2\) only, the admissible range of \(p\) in the last conclusion would be restricted to \(p>2\) as in \cite{CMS}.

\begin{proof}
Part (b) is immediate from the observation that \(C_q^{(\alpha)}(F)\leq M(A^{(\alpha)}(F)^q)^{1/q}\), where \(M\) is the Hardy--Littlewood maximal function, and the maximal inequality in \(L^{r}\), where \(r=p/q\in(1,\infty]\).

Part (a) follows from the distributional inequality
\begin{equation}\label{eq:goodLambda}
  \abs{\{A^{(\alpha)}(F)>2\lambda\}}
  \leq \abs{\{C_q^{(\alpha)}(F)>\gamma\lambda\}}
   + C\gamma^q\abs{\{A^{(\alpha+10)}(F)>\lambda\}},
\end{equation}
which, for all \(\alpha,q\in(0,\infty)\), is true with some \(C\), and for all \(\gamma\in(0,1],\lambda\in(0,\infty)\), as we show in the following lemma.
Integrating \eqref{eq:goodLambda} multiplied by \(p\lambda^{p-1}\), we obtain
\begin{equation*}
  2^{-p}\Norm{A^{(\alpha)}(F)}{L^p}^p
  \leq\gamma^{-p}\Norm{C_q^{(\alpha)}(F)}{L^p}^p+C\gamma^q\Norm{A^{(\alpha+10)}(F)}{L^p}^p.
\end{equation*}
Since \(\Norm{A^{(\alpha+10)}(F)}{L^p}\leq C'\Norm{A^{(\alpha)}(F)}{L^p}\) by Theorem~\ref{thm:aperture}, we obtain
\begin{equation*}
  \Norm{A^{(\alpha)}(F)}{L^p}\lesssim\Norm{C_q^{(\alpha)}(F)}{L^p}
\end{equation*}
provided that \(\gamma>0\) is chocen sufficiently small and the left-hand side is finite. The finiteness assumption can be removed by a standard approximation argument; cf.~\cite{CMS}.
\end{proof}

\begin{lemma}[\cite{CMS}, Lemma~3]
For all \(q,\alpha\in(0,\infty)\), there exists a constant \(C\) such that
\begin{equation*}
  \abs{\{A^{(\alpha)}(F)>2\lambda; C_q^{(\alpha)}(F)\leq\gamma\lambda\}}
  \leq C\gamma^q\abs{\{A^{(\alpha+10)}(F)>\lambda\}}
\end{equation*}
for all \(\gamma\in(0,1]\) and \(\lambda\in(0,\infty)\). In particular, \eqref{eq:goodLambda} holds for the same parameters.
\end{lemma}

\begin{proof}
Let \(\beta:=\alpha+10\).
Denoting by \(\bigcup_{k=1}^{\infty}Q_k\) a Whitney decomposition of the open set \(\{A^{(\beta)}(F)>\lambda\}\), it suffices to prove that
\begin{equation}\label{eq:inQk}
  \abs{\{x\in Q_k: A^{(\alpha)}(F)(x)>2\lambda, C_q^{(\alpha)}(F)(x)\leq\gamma\lambda\}}
  \leq C\gamma^q\abs{Q_k}
\end{equation}
for each \(k\). Thus we fix a particular \(Q_k\) and denote by \(B\) the minimal ball containing it.
By the properties of the Whitney decomposition, for each \(k\) there exists an \(x_k\in\{A^{(\beta)}(F)\leq\lambda\}\) such that \(d(x_k,Q_k)\leq 8r(B)\) and hence \(d(x_k,x)\leq 10r(B)\) for all \(x\in Q_k\).

We decompose
\begin{equation*}
  F(y,t)=F(y,t)1_{[r(B),\infty)}(t)+F(y,t)1_{(0,r(B))}(t)=:F_1(y,t)+F_2(y,t).
\end{equation*}
If \(x\in Q_k\) and \((y,t)\in\cone_{\alpha}(x)\) with \(t\geq r(B)\), then
\begin{equation*}
  \abs{y-x_k}\leq\abs{y-x}+\abs{x-x_k}<\alpha t+10 r(B)\leq (\alpha+10)t=\beta t,
\end{equation*}
and hence \((y,t)\in\cone_{\beta}(x_k)\). Then
\begin{equation*}
  A^{(\alpha)}(F_1)(x)\leq A^{(\beta)}(F)(x_k)\leq\lambda.
\end{equation*}
On the other hand, there holds \(A^{(\alpha)}(F_2)(x)=A^{(\alpha)}(F|r(B))(x)\), so that
\begin{equation*}
  \frac{1}{\abs{B}}\int_B A^{(\alpha)}(F_2)^q(y)\ud y
  \leq c \inf_{x\in B} C_q^{(\alpha)}(F)^q(x)\leq c(\gamma\lambda)^q,
\end{equation*}
where the last estimate is true whenever the set on the left of \eqref{eq:inQk} is non-empty. It follows that
\begin{equation*}
\begin{split}
  \abs{\{x\in Q_k:A^{(\alpha)}(F)(x)>2\lambda\}}
  &\leq\abs{\{x\in G_k:A^{(\alpha)}(F_2)(x)>\lambda\}} \\
  &\leq c(\gamma\lambda)^q\abs{B}\frac{1}{\lambda^q}=C\gamma^q\abs{Q_k},
\end{split}
\end{equation*}
which was to be proven.
\end{proof}

Finally we come to a duality inequality involving the \(A\) and \(C\) functionals. We follow closely the scalar-valued argument from \cite{Stein}, Section~IV.4.4, starting with auxiliary results for the following \emph{stopping time}:
Fix some \(q\in(0,\infty)\) and \(\rho>1\), and define \(\tau(x)\) by
\begin{equation*}
  \tau(x):=\sup\{\tau>0:A(F|\tau)(x)\leq
   \rho C_q(F)(x)\}.
\end{equation*}

\begin{lemma} For all balls \(B\), there holds
\(\abs{B\cap\{\tau>r(B)\}}\geq (1-\rho^{-q})\abs{B}\).
\end{lemma}

\begin{proof}
Observe that
\begin{equation*}
  \frac{1}{\abs{B}}\int_B A(F|r(B))^q(y)\ud y
  \leq\inf_{x\in B}C_q(F)^q(x),
\end{equation*}
so that
\begin{equation*}\begin{split}
  \inf_{x\in B}C_q(F)^q(y)
  &\geq\frac{1}{\abs{B}}\int_{B\cap\{\tau\leq r(B)\}}
    A(F|r(B))^q(y)\ud y \\
  &\geq\frac{1}{\abs{B}}\int_{B\cap\{\tau\leq r(B)\}}
    \rho^p C_q(F)^q(y)\ud y \\
  &\geq \rho^q\frac{\abs{B\cap\{\tau\leq r(B)\}}}{\abs{B}}
    \inf_{x\in B}C_q(F)^q(x).
\end{split}\end{equation*}
The claim follows after cancellation and complementation.
\end{proof}

\begin{corollary}\label{cor:Fubini}
For all positive measurable functions \(H\) on \(\reals^{n+1}_+\), there holds
\begin{equation*}
  \iint_{\reals^n}H(y,t)\frac{\ud y\ud t}{t}
  \lesssim\int_{\reals^n}\iint_{\cone^{\tau(x)}(x)}H(y,t)\frac{\ud y\ud t}{t^{n+1}}\ud x.
\end{equation*}
\end{corollary}

\begin{proof}
This is immediate from the lemma after changing the order of integration on the right.
\end{proof}

\begin{theorem}\label{thm:duality}
Let \(F:\reals^{n+1}_+\to X\), \(G:\reals^{n+1}\to X'\) be strongly measurable, and \(q\in(0,\infty)\). Then
\begin{equation*}
  \iint_{\reals^{n+1}_+}\abs{\pair{F(y,t)}{G(y,t)}}\frac{\ud y\ud t}{t}
  \lesssim\int_{\reals^n}C_q(F)(x)A(G)(x)\ud x.
\end{equation*}
\end{theorem}

\begin{proof}
By Corollary~\ref{cor:Fubini} and Proposition~\ref{prop:duality} (the duality of the Gauss norms),
\begin{equation*}\begin{split}
  LHS
  &\lesssim\int_{\reals^n}\iint_{\cone^{\tau(x)}(x)}\abs{\pair{F(y,t)}{G(y,t)}}
     \frac{\ud y\ud t}{t^{n+1}}\ud x \\
  &\leq\int_{\reals^n}\Norm{F\cdot 1_{\cone^{\tau(x)}(x)}}{\gamma(X)}
      \Norm{G\cdot 1_{\cone^{\tau(x)}(x)}}{\gamma(X')}\ud x \\
  &=\int_{\reals^n}A(F|\tau(x))(x)A(G|\tau(x))\ud x \\
  &\lesssim\int_{\reals^n}C_q(F)(x)A(G)(x)\ud x=RHS,
\end{split}\end{equation*}
where the last estimate used the defining property of the stopping time \(\tau(x)\).
\end{proof}

\section{The Carleson-type characterization of \(\BMO\)}\label{sec:charBMO}

We are now ready for the proof of Theorem~\ref{thm:charBMO}. It follows the general structure of the scalar-valued argument from \cite{Stein}, Sections IV.4.3--4, now that vector-valued versions of the same tools are available.

\begin{proof}[Proof of \(f\in\BMO\Rightarrow C_q(F)\in L^{\infty}\)]
Since \(C_q(F)(x)\) is a non-decreasing function of \(q\in(0,\infty)\), we may assume that \(q\in(1,\infty)\).
We fix a ball \(B\), split the function as
\begin{equation*}
  f=(f-f_{3B})1_{3B}+(f-f_{3B})1_{(3B)^c}+f_{3B}
  =:f_1+f_2+f_3,
\end{equation*}
and write \(F_i(x,t):=f_i*\psi_t(x)\).

For \(F_1\), we have
\begin{equation*}\begin{split}
  &\Big(\frac{1}{\abs{B}}\int_B A(F_1|r(B))^q(y)\ud y\Big)^{1/q}
   \leq\Big(\frac{1}{\abs{B}}\int_{\reals^n} A(F_1)^q(y)\ud y\Big)^{1/q} \\
  &=\abs{B}^{-1/q}\Norm{A(F_1)}{L^q(\reals^n)}\lesssim\abs{B}^{-1/q}\Norm{f_1}{L^q(\reals^n;X)}
    \lesssim\Norm{f}{\BMO(\reals^n;X)},
\end{split}\end{equation*}
where we used Theorem~\ref{thm:LpTp}, and hence the UMD property.

Concerning \(F_2\), we can estimate
\begin{equation}\label{eq:F2}\begin{split}
  &\Norm{F_2\cdot 1_{\cone^{r(B)}(y)}}{\gamma(X)} \\
  &=\BNorm{(u,t)\mapsto\int_{(3B)^c}[f(z)-f_{(3B)^c}]\psi_t(u-z)\ud z\, 1_{\cone^{r(B)}(y)}(u,t)}{\gamma(X)} \\
  &\leq\int_{(3B)^c}\Norm{(u,t)\mapsto[f(z)-f_{(3B)^c}]\psi_t(u-z) 1_{\cone^{r(B)}(y)}(u,t)}{\gamma(X)}\ud z \\
  &=\int_{(3B)^c}\norm{f(z)-f_{(3B)^c}}{X}\cdot\Norm{\psi_t(u-z) 1_{\cone^{r(B)}(y)}(u,t)}{
    L^2(\tfrac{\ud u\ud t}{t^{n+1}})}\ud z.
\end{split}\end{equation}
For \(z\in(3B)^c\), \(y\in B\), and \((u,t)\in\cone^{r(B)}(y)\), there holds
\begin{equation*}
\begin{split}
  \abs{z-u}
  &\geq\abs{z-c(B)}-\abs{c(B)-y}-\abs{y-u} \\
  &\geq\abs{z-c(B)}-2r(B)\geq 3^{-1}\abs{z-c(B)},
\end{split}
\end{equation*}
thus
\begin{equation*}
  \abs{\psi_t(u-z)}\lesssim t^{-n}\Big(1+\frac{\abs{u-z}}{t}\Big)^{-(n+1)}
  \lesssim \frac{t}{\abs{z-c(B)}^{n+1}},
\end{equation*}
and hence the \(L^2\) norm on the right of \eqref{eq:F2} may be bounded by
\begin{equation*}
  \Big(\int_0^{r(B)}\int_{B(x,t)}\abs{\psi_t(u-z)}^2\ud y \frac{\ud t}{t^{n+1}}\Big)^{1/2}
  \lesssim \frac{r(B)}{\abs{z-c(B)}^{n+1}}.
\end{equation*}

It follows that
\begin{equation*}\begin{split}
  &\Norm{F_2\cdot 1_{\cone^{r(B)}(y)}}{\gamma(X)}
  \lesssim\int_{(3B)^c}\norm{f(z)-f_{3B}}{X}\frac{r(B)}{\abs{z-c(B)}^{n+1}}\ud z
  \lesssim\Norm{f}{\BMO(X)},
\end{split}\end{equation*}
for all \(y\in B\), and in particular the same upper bound holds for the \(L^q\) average of the left-side quantity over the ball \(B\). This argument for \(F_2\), as a matter of fact, did not employ any special properties of the Banach space \(X\).

Finally, since \(\int\psi(y)\ud y=0\) we have
\(f_3*\psi_t(x)\equiv 0\), so the estimate for \(F_3\) is trivial.
\end{proof}

\begin{proof}[Proof of \(C_q F\in L^{\infty}\Rightarrow f\in\BMO\)]
We establish this by invoking the (vector-valued) \(H^1\)--\(\BMO\)-duality. 
Consider a function \(g\) in the following dense subspace of the Hardy space \(H^1(\reals^n;X')\): We assume that \(g\) is compactly supported, bounded, and takes its values in a finite-dimensional subspace of \(X'\).
Due to the non-degeneracy of \(\psi\), there is a complementary Schwartz function \(\phi\) such that \eqref{eq:complementary} holds. Then we have
\begin{equation*}
  \int_{\reals^n}\pair{f(x)}{g(x)}\ud x
  =\int_0^{\infty}\int_{\reals^n}
    \pair{f*\psi_t(x)}{g*\phi_t(x)}\ud x\frac{\ud t}{t},
\end{equation*}
which identity, for our choice of \(g\), is a simple consequence of
the corresponding result for scalar functions \(f\) and \(g\)
found in~\cite{Stein}, Sections~IV.4.4.1--2. Denoting \(G(y,t):=g*\phi_t(y)\), it follows from Theorem~\ref{thm:duality} and Corollary~\ref{cor:H1T1} that
\begin{equation*}\begin{split}
  &\Babs{\int_{\reals^n}\pair{f(x)}{g(x)}\ud x}
  \lesssim\int_{\reals^n}C_q(F)(x)A(G)(x)\ud x \\
  &\leq\Norm{C_q(F)}{L^{\infty}(\reals^n)}\Norm{A(G)}{L^1(\reals^n)}
   \lesssim\Norm{C_q(F)}{L^{\infty}(\reals^n)}\Norm{g}{H^1(\reals^n;X')}.
\end{split}\end{equation*}
Thus \(f\) acts as a bounded functional on \(H^1 (\reals^n;X')\) via the natural duality, which implies that \(f\in\BMO(\reals^n;X)\), since \(H^1 (\reals^n;X')'\eqsim\BMO(\reals^n;X)\) when the Banach space \(X\) is reflexive, in particular when it is UMD (see~\cite{Blasco}).
\end{proof}

\section{Carleson embedding and paraproducts}\label{sec:Carleson}

In this section we first develop some further inequalities involving the \(A\), \(C\) and \(N\) functionals, which are then applied to prove Theorem~\ref{thm:paraproduct}. The inequality behind that theorem is contained in the ``Carleson embedding theorem'', which we give next.

In its statement we use the notion of \emph{type} \(q\in(0,2]\) of a Banach space \(Y\), defined by the requirement that
\begin{equation}\label{eq:type}
  \Exp_{\radem}\bnorm{\sum_{k=1}^N\radem_k y_k}{Y}
  \lesssim\Big(\sum_{k=1}^N\norm{y_k}{Y}^q\Big)^{1/q}
\end{equation}
for all \(N\in\ints_+\) and \(y_1,\ldots,y_N\in Y\), where \(\radem_k\) are independent Rademacher variables (i.e., distributed according to the law \(\prob(\radem_k=+1)=\prob(\radem_k=-1)=1/2\)) and \(\Exp_{\radem}\) the corresponding mathematical expectation.  

The estimate \eqref{eq:type} is trivial for \(q\in(0,1]\), and in fact this case of Theorem~\ref{thm:Carleson} would suffice for the proof of Theorem~\ref{thm:paraproduct}. However, we have chosen to consider a general \(q\in(0,2]\) in order to allow easier comparison to the scalar-valued case, where \(\complex\) does have type \(2\) and moreover the Carleson functional \(C_2\) is the classical choice. Indeed, the case \(X=Y=\bddlin(X,Y)=\complex\) and \(q=2\) of the following theorem is given in \cite{CMS}, Remark (b) on p.~320.

In the proof, we will need the type $q$ property of the space $\gamma(H,Y)$ instead of that of~$Y$ but these are actually equivalent. Indeed, this follows from the well-known fact that $L^2(\Omega;Y)$ inherits the type $q$ property of $Y$, and $Y$ is isomorphic to a subspace of $\gamma(H,Y)$, which is isomorphic to a subspace of $L^2(\Omega,Y)$.

\begin{theorem}\label{thm:Carleson}
Let \(F:\reals^{n+1}_+\to X\), \(G:\reals^{n+1}_+\to\bddlin(X,Y)\) be measurable.
Let \(\beta>\alpha>0\), \(0<q<p\leq\infty\) and \(Y\) have type \(q\). Then
\begin{equation*}
  \Norm{C_q^{(\alpha)}(G\cdot F)}{L^p(\reals^n)}
  \lesssim\Norm{N^{(\beta)}(G)}{L^p(\reals^n)}
  \Norm{C_q^{(\alpha)}(F)}{L^{\infty}(\reals^n)}.
\end{equation*}
\end{theorem}

In the vector-valued theory, Theorem~\ref{thm:Carleson} has a certain analogy to \cite{HMP}, Theorem~8.2, and the basic ingredient of the proof, contained in the following lemma, is related to and inspired by \cite{HMP}, Lemma~8.1. However, once the lemma is established, the way it is used to obtain Theorem~\ref{thm:Carleson} differs from the approach in \cite{HMP}, the most notable departure being that we are able to avoid any use of interpolation.

\begin{lemma}
Let \(F:\reals^{n+1}_+\to X\) and \(G:\reals^{n+1}_+\to\bddlin(X,Y)\) be measurable.
Let \(\beta>\alpha>0\) and \(q\in(0,2]\) be such that \(Y\) has type \(q\). Then for all balls \(B\) there holds
\begin{equation*}
  \int_B A^{(\alpha)}(G\cdot F|r(B))^q(x)\ud x
  \lesssim\int_{(1+\alpha+\beta)B} N^{(\beta)}(G)^q(x)\ud x\cdot\Norm{C_q^{(\alpha)}(F)}{L^{\infty}(\reals^n)}^q.
\end{equation*}
\end{lemma}

\begin{proof}
Let \(G_0:=G\cdot 1_{(1+\alpha)B\times(0,r(B))}\), so that for all \(x\in B\) we have
\begin{equation*}
  A^{(\alpha)}(F\cdot G|r(B))(x)= A^{(\alpha)}(F\cdot G_0)(x).
\end{equation*}
We introduce the stopping times
\begin{equation*}
  \tau_k(x):=\inf\{\tau>0:\gbound\{\abs{G_0(y,t)}:(y,t)\in\cone_{\alpha}(x);t\geq\tau\}\leq 2^k\}.
\end{equation*}
Then we can estimate
\begin{equation*}\begin{split}
  A^{(\alpha)}(F\cdot G_0)(x)
  &=\BNorm{\sum_{k\in\ints}F\cdot G_0\cdot
    1_{\cone_{\alpha}^{\tau_{k}(x)}(x)\setminus\cone_{\alpha}^{\tau_{k+1}(x)}(x)}}{\gamma(Y)} \\
  &=\Exp_{\radem}\BNorm{\sum_{k\in\ints}\radem_k
    F\cdot G_0\cdot 1_{\cone_{\alpha}^{\tau_{k}(x)}(x)\setminus\cone_{\alpha}^{\tau_{k+1}(x)}(x)}}{\gamma(Y)} \\
  &\lesssim\Big(\sum_{k\in\ints}\Norm{F\cdot G_0\cdot
     1_{\cone_{\alpha}^{\tau_{k}(x)}(x)\setminus\cone_{\alpha}^{\tau_{k+1}(x)}(x)}}{\gamma(Y)}^q\Big)^{1/q} \\
  &\lesssim\Big(\sum_{k\in\ints}2^{kq}\Norm{F\cdot 1_{\cone_{\alpha}^{\tau_k(x)}(x)}}{\gamma(Y)}^q\Big)^{1/q},
\end{split}\end{equation*}
where the type \(q\) assumption was used in the second to last step, and the definition of the stopping times in the last one. The introduction of the Rademacher variables \(\radem_k\) in the second step is in effect the invariance of the Gauss norm under the multiplication by a unimodular function; this step can be omitted for \(q\in(0,1]\).
Integrating over the ball \(B\), it follows that
\begin{equation*}\begin{split}
  \int_B A^{(\alpha)}(F\cdot G|r(B))^q(x)\ud x
  &\leq\int_{\reals^n} A^{(\alpha)}(F\cdot G_0)^q(x)\ud x \\
  &\lesssim\sum_{k\in\ints}\int_{\reals^n} 2^{kq}
       A^{(\alpha)}(F|\tau_k(x))^q(x) \ud x \\
  &\leq\sum_{k\in\ints}\int_{\{\tau_k>0\}} 2^{kq}
       A^{(\alpha)}(F|\tau_k(x))^q(x) \ud x.
\end{split}\end{equation*}

Observe that \(\{\tau_k>0\}=\{N^{(\alpha)}G_0>2^k\}\subseteq\{N^{(\beta)}G>2^k\}\) since \(\beta>\alpha\). Let us denote by \(\bigcup_{j=1}^{\infty}Q_{kj}\) a Whitney decomposition of the open set \(\{N^{(\beta)}G>2^k\}\). Then
\begin{equation*}\begin{split}
  \int_{\{\tau_k>0\}} 2^{kq} A^{(\alpha)}(F|\tau_k(x))^q(x) \ud x
  \leq 2^{kq}\sum_{j=1}^{\infty}\int_{Q_{kj}} A^{(\alpha)}(F|\tau_k(x))^q(x)\ud x
\end{split}\end{equation*}

By the properties of the Whitney decomposition, for each \(Q_{kj}\) we have
\begin{equation*}
 d(Q_{kj},\{N^{(\beta)}G_0\leq 2^k\})\leq 4\diam(Q_{kj}). 
\end{equation*}
We may hence pick an \(x_{kj}\) such that \(N^{(\beta)}G_0(x_{kj})\leq 2^k\) and \(\abs{x-x_{kj}}\leq 5\diam(Q_{kj})\) for all \(x\in Q_{kj}\). If \(x\in Q_{jk}\), \((y,t)\in\cone_{\alpha}(x)\), and \(t\geq 5\diam(Q_{jk})(\beta-\alpha)^{-1}\), then
\begin{equation*}
  \abs{y-x_{kj}}\leq\abs{y-x}+\abs{x-x_{kj}}<\alpha t+5\diam(Q_{kj})\leq\beta t,
\end{equation*}
and hence \((y,t)\in\cone_{\beta}(x_{kj})\). It follows that
\begin{equation*}
  \gbound\{G_0(y,t):(y,t)\in\cone_{\alpha}(x),t\geq 5\diam(Q_{jk})(\alpha-1)^{-1}\}
  \leq 2^k,
\end{equation*}
and hence \(\tau_k(x)\leq 5\diam(Q_{jk})(\beta-\alpha)^{-1}\). This implies that
\begin{equation*}
\begin{split}
  \int_{Q_{kj}} A^{(\alpha)}(F|\tau_k(x))^q(x)\ud x
  &\leq\int_{Q_{kj}} A^{(\alpha)}(F|5\diam(Q_{jk})(\beta-\alpha)^{-1})^q(x)\ud x \\
  &\lesssim \abs{Q_{kj}}\Norm{C_q^{(\alpha)}(F)}{L^{\infty}}^q.
\end{split}
\end{equation*}
Substituting back we obtain
\begin{equation*}\begin{split}
  \int_B A^{(\alpha)}(F\cdot G|r(B))^q(x)\ud x 
  &\lesssim\sum_{k\in\ints}2^{kq}\sum_{j=1}^{\infty}\abs{Q_{kj}}\cdot\Norm{C_q^{(\alpha)}(F)}{L^{\infty}}^q \\
  &=\sum_{k\in\ints}2^{kq}\abs{\{N^{(\beta)}G_0>2^k\}}\cdot\Norm{C_q^{(\alpha)}(F)}{L^{\infty}}^q \\
  &\eqsim\Norm{N^{(\beta)}G_0}{L^q}^q\Norm{C_q^{(\alpha)}(F)}{L^{\infty}}^q.
\end{split}\end{equation*}

Finally, we note that if \(d(x,B)\geq(\alpha+\beta)r(B)\), then \(\cone_{\beta}(x)\cap[(1+\alpha)B\times(0,r(B))]=\emptyset\) and hence \(N^{(\beta)}G_0(x)=0\). On the other hand, we always have \(N^{(\beta)}G_0(x)\leq N^{\beta}G(x)\). These observations imply that
\begin{equation*}
  \Norm{N^{(\beta)}G_0}{L^q(\reals^n)}
  \leq\Norm{N^{(\beta)}G}{L^q((1+\alpha+\beta)B)}
\end{equation*}
and complete the proof.
\end{proof}

Theorems~\ref{thm:Carleson} and \ref{thm:paraproduct} are now easy consequences.

\begin{proof}[Proof of Theorem~\ref{thm:Carleson}]
For each \(x\in\reals^n\) we have, using the previous lemma,
\begin{equation*}\begin{split}
  C_q^{(\alpha)}(G\cdot F)(x)
  &=\sup_{B\owns x}\Big(\frac{1}{\abs{B}}\int_B A^{(\alpha)}(G\cdot F|r(B))^q(y)\ud y\Big)^{1/q} \\
  &\lesssim\sup_{B\owns x}\Big(\frac{1}{\abs{B}}\int_{(1+\alpha+\beta)B}N^{(\beta)}(G)^q(y)\ud y\Big)^{1/q}
    \Norm{C_q^{(\alpha)}(F)}{L^{\infty}} \\
  &\lesssim M(N^{(\beta)}(G)^q)^{1/q}(x)\Norm{C_q^{(\alpha)}(F)}{L^{\infty}}.
\end{split}\end{equation*}
The proof is concluded by an application of the maximal inequality in \(L^r\) with \(r=p/q\in(1,\infty]\).
\end{proof}

\begin{proof}[Proof of Theorem~\ref{thm:paraproduct}]
We are given three functions
\begin{equation*}
  f\in\BMO(\reals^n;X),\qquad
  u\in L^p(\reals^n),\qquad
  g\in L^{p'}(\reals^n;X').
\end{equation*}
Let us define
\begin{equation*}
  F(x,t):=f*\psi_t(x),\qquad
  U(x,t):=u*\phi_t(x),\qquad
  G(x,t):=g*\tilde\psi_t(x),
\end{equation*}
where \(\tilde\psi(x):=\psi(-x)\).
By the obvious identification of \(\lambda\in\complex\) and \(\lambda I\in\bddlin(X)\), we think of \(U\) as an operator-valued function. Let \(\alpha>1\). Then we use Theorem~\ref{thm:duality} and Theorem~\ref{thm:Carleson} to obtain:
\begin{equation*}\begin{split}
  \iint_{\reals^{n+1}_+}&\abs{\pair{U(x,t)F(x,t)}{G(x,t)}}\frac{\ud x\ud t}{t}
  \lesssim\Norm{C_q(U\cdot F)}{L^{p}(\reals^n)}\Norm{A(G)}{L^{p'}(\reals^n)} \\
  &\lesssim\Norm{N^{(\alpha)}(U)}{L^p(\reals^n)}\Norm{C_q(F)}{L^{\infty}(\reals^n)}\Norm{A(G)}{L^{p'}(\reals^n)}.
\end{split}\end{equation*}
Recall that
\begin{equation*}
  \Norm{C_q(F)}{L^{\infty}(\reals^n)}\lesssim\Norm{f}{\BMO(\reals^n;X)},\qquad
  \Norm{A(G)}{L^{p'}(\reals^n)}\lesssim\Norm{g}{L^{p'}(\reals^n;X')}
\end{equation*}
by Theorem~\ref{thm:charBMO} and Theorem~\ref{thm:LpTp}. Finally, we observe the pointwise domination  \(N^{(\alpha)}(U)\lesssim M(u)\), from which
\begin{equation*}
  \Norm{N^{(\alpha)}(U)}{L^p(\reals^n)}
  \lesssim\Norm{M(u)}{L^p(\reals^n)}
  \lesssim\Norm{u}{L^p(\reals^n)}
\end{equation*}
follows by the maximal inequality.
This shows the convergence of the integral defining \(\pair{P(f,u)}{g}\) as well as the asserted norm estimate for \(P(f,u)\).
\end{proof}

It seems worthwhile formulating one more consequence of Theorem~\ref{thm:Carleson} in terms of the \(A\) and \(C\) functionals:

\begin{corollary}
Let \(F:\reals^{n+1}_+\to X\), \(G:\reals^{n+1}_+\to\bddlin(X,Y)\) be strongly measurable.
Let \(Y\) be a UMD space, \(q\in(0,\infty)\), and \(p\in(1,\infty)\). Then
\begin{equation*}
  \Norm{A(G\cdot F)}{L^p(\reals^n)}
  \lesssim 
  \Norm{N(G)}{L^p(\reals^n)}\Norm{C_q(F)}{L^{\infty}(\reals^n)}.
\end{equation*}
\end{corollary}

\begin{proof}
Since \(C_q(F)\) is non-decreasing in \(q\in(0,\infty)\), we may again assume that \(q\in(0,1]\). Then \(Y\) has automatically type \(q\). Let \(\alpha\in(0,1)\). We apply Theorem~\ref{thm:aperture}, Theorem~\ref{thm:AC}, and Theorem~\ref{thm:Carleson} with \(\beta=1>\alpha>0\), in this order, to get
\begin{equation*}
  \Norm{A(G\cdot F)}{L^p}
  \lesssim\Norm{A^{(\alpha)}(G\cdot F)}{L^p}
  \lesssim\Norm{C^{(\alpha)}_q(G\cdot F)}{L^p}
  \lesssim\Norm{N(G)}{L^p}\Norm{C_q(F)}{L^{\infty}},
\end{equation*}
and this was the claim.
\end{proof}



\end{document}